\documentclass[reqno]{amsart}
\usepackage{amsmath, amsthm, amssymb, enumitem, mathrsfs}
\usepackage{hyperref,soul, iitem}
\usepackage{color}
\usepackage{cite}
\usepackage{setspace}
\usepackage{adjustbox}
\usepackage{tikz}

\newcommand{\ostar}{%
  \mathbin{%
    \tikz[baseline=(s.base)]{
      \node (s) at (0,0) {$\star$};
      \draw[line width=0.4pt] (s.center) circle[radius=0.82ex];
    }%
  }%
}

\calclayout

\setenumerate{label=\textnormal{(\arabic*)}} 

\newcommand{\pow}{\mathcal{P}}
\newcommand{\cat}{\mathsf}

\numberwithin{equation}{section}

\theoremstyle{plain}
\newtheorem{theorem}[equation]{Theorem}
\newtheorem{corollary}[equation]{Corollary}
\newtheorem{lemma}[equation]{Lemma}
\newtheorem{proposition}[equation]{Proposition}

\theoremstyle{definition}
\newtheorem{definition}[equation]{Definition}

\newtheorem{remark}[equation]{Remark}

\begin{document}

\begin{spacing}{1.05}

\title{Hyperstructures as complete atomic Boolean algebras}

\author{Ariel E. Rosenfield}
\email{ariari@umd.edu}
\address{Department of Mathematics\\ University of Maryland---College Park \\
2100E IPST\\ College Park, MD 20742\\ USA}

\date{\today}

\begin{abstract} We show that several categories of hyperstructures, including the mosaics of Nakamura and Reyes and the hypergroups of Marty, are realizable as categories of objects with extra structure within the category of complete atomic Boolean algebras (CABAs) with atom-preserving morphisms. 
\end{abstract}

\maketitle

\setcounter{tocdepth}{1} 
\makeatletter
\def\l@subsection{\@tocline{2}{0pt}{2.5pc}{5pc}{}}
\def\l@subsubsection{\@tocline{2}{0pt}{5pc}{7.5pc}{}} 
\makeatother

\section{Introduction}

The fundamental objects of study in any elementary algebra course, such as monoids, groups, rings, and modules, are expressible as objects endowed with extra structure in an appropriate base category---for example, an abelian group is a commutative monoid with inverses in the category of sets, and a ring is a monoid object in the category of abelian groups. Since hypersemigroups and hypermonoids---analogues of semigroups and monoids in which a binary operation $X \times X \to X$ on a set $X$ is replaced with a function $X \times X \to \pow(X)$ into the power set of $X$---are algebraic structures, it seems natural to wonder whether one can view them as ``objects with extra structure" in an appropriate category.

In \cite[N]{Dudzik_2016}, it is shown that we can view hypersemigroups and hypermonoids in this way, in the sense that there is a bijective correspondence between, for example, hypersemigroup structures on a set $X$ and semigroup structures on $\pow(X)$, viewed as an object of the category of suplattices. 

Below, we prove analogous results for the mosaics of \cite{mosaics} and the hypergroups of \cite{marty}, and extend these bijections to equivalences of categories. We take $\cat{CABA}$ to denote the category of complete atomic Boolean algebras (CABAs) with suplattice morphisms between them. Letting $\cat{Msc}$ denote the category of mosaics, and letting $\mathcal{O}_\mathsf{Msc}(\cat{CABA})_A$ denote the category of unital residuated CABAs with atom-preserving morphisms between them, we prove several results of the following type:

\begin{theorem}
    There is an equivalence of categories $\cat{Msc} \cong \mathcal{O}_\mathsf{Msc}(\cat{CABA})_A$, and a pseudomonic embedding $$\mathcal{O}_\mathsf{Msc}(\cat{CABA})_A \hookrightarrow \mathcal{O}_\mathsf{Msc}(\cat{CABA}).$$
\end{theorem}

\subsection{Acknowledgements} This work constitutes part of a joint paper in progress together with Florian Warg and Tony Zorman. Thanks to Yiqi Xu, So Nakamura, and Manny Reyes for their insights and feedback. The author attests that no generative AI tools were used in the production of this work.

\section{Preliminaries} By a \textbf{suplattice}, we mean a partially ordered set admitting arbitrary joins, including the nullary join $\bot$. Note that any such lattice is in fact complete, i.e., it also admits arbitrary meets. By a \textbf{morphism of suplattices}, we mean a join-preserving, but not necessarily meet-preserving, monotone set function from the underlying set of one lattice to the underlying set of another. We denote the category of suplattices and suplattice morphisms by $\cat{SupLat}$.

By an \textbf{atom} of a suplattice $L$, we mean an element $x \in L$ such that $\bot \neq x$, and such that the downset $x^\downarrow$ of $x$ is $\{x, \bot\}$. We denote the set of atoms of a lattice $L$ by $\cat{At}(L)$. By an \textbf{atomic} suplattice, we mean a suplattice $L$ such that for any $y \in L$, $y^\downarrow \cap \cat{At}(L)$ is nonempty.

A \textbf{Boolean algebra} is a complemented distributive lattice: A lattice $(L, \land, \lor, \bot)$ equipped with a unary operation $(-)^c$, called \textbf{complementation}, such that $x^c \land x = \bot$ and $x^c \lor x = L$ for any $x \in L$, and such that finitary $\land$ and $\lor$ distribute over one another. A \textbf{complete Boolean algebra} is a suplattice admitting a Boolean algebra structure, and a \textbf{complete atomic Boolean algebra (CABA)} is an atomic suplattice admitting a Boolean algebra structure. We denote the category of these latter, with morphisms simply taken to be suplattice morphisms, by $\cat{CABA}$.

Recall that the objects of $\cat{CABA}$ are, up to isomorphism, exactly those suplattices arising as power sets. We will therefore frequently denote an arbitrary object of $\cat{CABA}$ by $\pow(X)$, where $X$ denotes an arbitrary set.

\subsection{Atoms and morphisms in CABA} We begin by recalling some elementary facts about the interaction between suplattice morphisms and atoms in CABAs. 

Say that a suplattice morphism $f : L \to M$ is \textbf{atom-preserving} if for any $y \in \text{At}(L)$, we have $f(y) \in \text{At}(M)$. It is easily checked that a composite of atom-preserving morphisms is again atom-preserving. It follows from \cite[Prop. 1]{Crapo_1967} that any surjective suplattice morphism is atom-preserving, and so the identity morphism on a given suplattice $X$ is atom-preserving. Thus, complete atomic Boolean algebras with atom-preserving suplattice morphisms form a category, which we denote by $\cat{CABA}_A$.

Recalling that an isomorphism of suplattices is a bijective (hence, a fortiori, surjective) suplattice morphism, it also follows from \cite[Prop. 1]{Crapo_1967} that:

\begin{proposition}
    \label{prop:isomorphism-atom-preserving}
    Any isomorphism of atomic suplattices is atom-preserving.
\end{proposition} 

By a \textbf{magma object} in $\cat{SupLat}$ (respectively in $\cat{CABA}$), we mean a suplattice (respectively, a CABA) $L$ equipped with a binary operation $\star : L \times L \to L$. By a \textbf{morphism of magma objects}, we mean a colax morphism---a morphism $f: (L, \star_L) \to (M, \star_M)$ of suplattices such that for any $x,y \in L$, we have $f(x \star_L y) \leq_Y f(x) \star_M f(y)$. We say that a magma operation $\star$ is \textbf{associative} if we have $\star \circ (\text{id} \times \star) = \star \circ (\star \times \text{id})$, and \textbf{unital} if there is an element $e \in L$ such that $\star(e,x) = x = \star(x,e)$ for all $x \in L$.

It is easily checked that a composite of colax morphisms is colax, and that the identity morphism on a given suplattice is colax. We therefore use the notation $\cat{Mag}(\cat{CABA})_A$ to denote the category of magma objects in $\cat{CABA}$ together with the atom-preserving colax suplattice morphisms, and use the notation ${\cat{Mag}}(\pow(X),\pow(Y))_A$ for an arbitrary hom-set in this category.

\subsection{Hyperstructures} By an $n$-ary \textbf{hyperoperation} on a set $X$, we mean a set function $\star : X^n \to \pow(X)$, where $\pow(X)$ denotes the power set of $X$. We recall some properties of binary hyperoperations.

\begin{definition}[\textit{Properties of binary hyperoperations}] \label{def:hyperop}  Let $X$ be any set. A binary hyperoperation ${\star : X^2 \to \pow(X)}$ may have the following properties:

\begin{itemize}
    \item[i.] It is \textbf{total} if $X$ is nonvoid and $x \star y$ is nonvoid for all $x,y \in X$;

    \item[ii.] It is \textbf{associative} if $$ x \star (y \star z) = (x \star y) \star z \qquad \forall x,y,z \in X; $$

    \item[iii.] It is \textbf{commutative} if $$ x \star y = y \star x \qquad \forall x,y \in X;$$

    \item[iv.] It is \textbf{weakly unital} if there is an atom $e \in \pow(X)$ (i.e., a singleton subset of $X$) such that $$ x \in e \star x \cap x \star e \qquad \forall x \in X;$$

    \item[v.] It is \textbf{unital} or has an \textbf{identity} if it weakly unital and we have $$x \star e = x = e \star x \qquad \forall x \in X;$$ 

    \item[vi.] It \textbf{has inverses} if it is unital and is equipped with a function $(-)' : X \to X$ such that $$ e \in [x \star x'] \cap [x' \star x] \qquad \forall x \in X;$$

    \item[vii.] It is \textbf{reversible} if it has inverses and satisfies $$ x \in y \star z \implies (y \in x \star z') \land (z \in y' \star x) \qquad \forall x,y,z \in X.$$
\end{itemize}
\end{definition} 

For expository ease in the proof of Theorem \ref{thm:is-relation-algebra} below, we note (following \cite{mosaics}) that in case $(X, \star)$ is unital with inverses, reversibility of $X$ is equivalent to the condition that for all triples $x,y,z \in X$, we have \begin{equation*} 
    x \in y \star z \iff y \in x \star z' \iff z \in y' \star x.
\end{equation*} We can formulate this alternatively as follows:

\begin{remark} \label{rem:charfns}
    A unital hypermagma $X$ with inverses is reversible if and only if for every triple $(y,z,x) \in X^3$, we have \begin{equation*}
    (y \star z) \cap \{x\} = \varnothing \iff ( x \star z')\cap \{y\} = \varnothing \iff (y' \star x) \cap \{z\} = \varnothing.
\end{equation*} 
\end{remark}

We will call a set equipped with a binary hyperoperation $(X,\star)$ a \textbf{hypermagma}, and say that such an $X$ is an object of the category $\cat{HMag}$, whose morphisms are defined as follows: 

\begin{definition} \label{def:colaxity-unitality}
    A function $f : (X,\star_X) \to (Y,\star_Y)$ between hypermagmas is called \textbf{colax} if $$ f(x \star_X y) \subseteq f(x) \star_Y f(y) \qquad \forall x,y \in X. $$ Take $\cat{HMag}(X,Y)$ to be the set of colax functions between hypermagmas. We say that a colax map $f$ is \textbf{unital} if $X$ has an identity $1_X$ whose image $f(1_X)$ is an identity for $Y$. 
\end{definition}

We will discuss many layers of structure on a given hypermagma. To keep track of these, we introduce the following notation for the various categories of hyperstructures we are interested in. In each case, we take the morphisms in a given category to be colax hypermagma morphisms, which we require to be unital if the objects of the category come equipped with an identity element (for example, morphisms in the category of hypermonoids are taken to be unital).

\begin{definition}[\textit{Categories of hyperstructures}] \label{def:gloss-hyper} We say that a set $X$ is:

\begin{itemize}

    \item ($\cat{HMag}$) A hypermagma if it admits a binary hyperoperation $\star$;
    
    \item ($\cat{uHMag}$) A unital hypermagma if it is a hypermagma and $\star$ is unital;

    \item ($\cat{HSemi}$) A \textbf{hypersemigroup} if it is an associative hypermagma;

    \item ($\cat{HMon}$) A \textbf{hypermonoid} if it is an associative unital hypermagma;

    \item ($\cat{Msc}$, $\cat{cMsc}$) A \textbf{(commutative)} \textbf{mosaic} if it is a (commutative) reversible hypermagma with identity---note that these structures are not required to be associative;

    \item ($\cat{HGrp}$) A \textbf{hypergroup} if it is a total, reversible hypermonoid;

    \item ($\cat{Can}$) A \textbf{canonical hypergroup} if it is a commutative hypergroup.
\end{itemize}
\end{definition}

\subsection{Binary operations from hyperoperations} We recall the following well-known correspondence (described in \cite[N]{Dudzik_2016}) between $n$-ary hyperoperations and $n$-ary operations on objects of $\cat{SupLat}$.

\begin{theorem} \label{thm:lattice-ops}
    For non-negative integers $n$, there are natural bijections $$ \cat{Set}(A^n,\mathcal{U}(B)) \cong \cat{SupLat}(\pow(A)^{\otimes n},B) \cong \cat{Mult}_\cat{SupLat}(\pow(A)^n, B),$$ where $\mathcal{U}$ is the forgetful functor on suplattices.
\end{theorem}

As a particular case of Theorem \ref{thm:lattice-ops}, given sets $X,Y$ and a nonnegative integer $n$, we have a bijection \begin{equation*}
    \cat{Set}(X^n,\pow(Y)) \cong \cat{Mult}_\cat{SupLat}(\pow(X)^n, \pow(Y))
\end{equation*} whose effect is to send a function $\star : X^n \to \pow(Y)$ to the suplattice morphism \begin{equation} \label{eq:hyperop-to-binop} \ostar(S_1,...,S_n) := \bigvee_{x_k \in S_k} \star(x_1,...,x_n) \end{equation} on $\pow(X)^n$. In the opposite direction, we obtain a function $\star : X^n \to \pow(Y)$ by restricting a multimorphism ${\ostar : \pow(X)^{n} \to \pow(Y)}$ to the singletons in $\pow(X)$; i.e., we define \begin{equation} \label{eq:binop-to-hyperop}
    \star(x_1,...,x_n) := \ostar(\{x_1\},...,\{x_n\}).
\end{equation}

\section{Hyperstructures as objects in CABA} We proceed to construct a functor on $\cat{HMag}$ which, when restricted to a given category $\mathcal{H}$ of hyperstructures, yields an equivalence of categories between $\mathcal{H}$ and the appropriate subcategory of $\cat{Mag}(\cat{CABA})_A$. We start by giving the construction for hypermagmas and magma objects.

\begin{proposition} \label{prop:hmag-mag-morphisms}
    There is a bijection $$\Phi : \cat{HMag}(X,Y) \cong \cat{Mag}(\pow(X),\pow(Y))_A$$ which moreover preserves composites and identities.
\end{proposition}

\begin{proof}
    Let $f : (X, \star_X) \to (Y, \star_Y)$ be a morphism of hypermagmas, and define a map $${\hat{f} : (\pow(X),\ostar_X) \to (\pow(Y), \ostar_Y)}$$ by $S \mapsto f(S)$; i.e., $\hat{f}$ takes the direct image of a set under $f$, and $\ostar$ is the binary operation of \ref{eq:hyperop-to-binop}. Since the direct image preserves subset inclusions and unions, $\hat{f}$ is a morphism of suplattices. It is easily checked, using properties of the direct image, that $\widehat{fg} = \widehat{f} \circ \widehat{g}$, and that $\widehat{\text{id}_X} = \text{id}_{\pow(X)}$. Moreover we observe, because $f$ is a function, that $\hat{f}$ is an atom-preserving lattice morphism---if $\{x\}$ is an atom of $\pow(X)$, then $\hat{f}(\{x\}) = \{f(x)\}$ is a singleton, whence an atom of $\pow(Y)$.
    
    To see that $\hat{f}$ is a colax morphism of magmas, note that since $f$ is a (colax) morphism of hypermagmas, we have for any $A,B \subset X$ that $$ f\left( A \ostar_X B \right) = \bigcup_{\substack{a \in A\\ b \in B}} f(a \star_X b) \subset \bigcup_{\substack{a \in A\\ b \in B}} f(a) \star_Y f(b) = f(A) \ostar_Y f(B).$$ 

    Conversely, if ${F : (\pow(X),\ostar_X) \to (\pow(Y), \ostar_Y)}$ is any atom-preserving suplattice morphism, we let $s(F)(x)$ denote the unique element of $F(\{x\})$. From this assignment, we obtain a well-defined function $s(F): X \to Y$. This extends to a function $s(F): (X, \star_X) \to (Y, \star_Y)$ between hypermagmas, where $\star$ is the hypermagma operation of \ref{eq:hyperop-to-binop}, by setting $$s(F)(x \star_X x') = F(\{x\} \ostar_X \{x'\}).$$ If $F$ is a colax morphism of magmas, we then have $$s(F)(x \star_X x') = F(\{x\} \ostar_X \{x'\}) \subset F(\{x\}) \ostar_Y F(\{x'\}) = s(F)(x) \star_Y s(F)(x'),$$ whence $s(F)$ is a colax hypermagma morphism.

    Finally, we show that the assignments $f \mapsto \hat{f}$ and $F \mapsto s(F)$ are inverse to one another. For a suplattice morphism ${F : (\pow(X),\ostar_X) \to (\pow(Y), \ostar_Y)}$, we have $\widehat{s(F)} = F$: If $A \subset X$, we have $$\widehat{s(F)}(A) = s(F)(A) = \bigcup_{a \in A} s(F)(a) = \bigcup_{a \in A} F(\{a\}) = F(A).$$ Conversely, for a hypermagma morphism $f : (X, \star_X) \to (Y, \star_Y)$, we have $s(\hat{f}) = f$: If $x \in X$, we have $s(\hat{f})(x) = \hat{f}(\{x\}) = f(x)$.
\end{proof}

Recalling that quantales are exactly semigroup objects (i.e., the associative magmas) in $\cat{SupLat}$, we use the notation $\cat{Quantale}(\pow(X),\pow(Y))_A$ to denote the set of colax non-singular magma morphisms between associative magmas in $\cat{CABA}$.

\begin{lemma} \label{lem:hsemi-semi-morphisms}
    The assignment $\Phi$ of Proposition \ref{prop:hmag-mag-morphisms} restricts to a bijection $$ \cat{HSemi}(X,Y) \cong \cat{Quantale}(\pow(X),\pow(Y))_A $$ which preserves composites and identities.
\end{lemma}

\begin{proof}
    That the correspondence \ref{eq:hyperop-to-binop} preserves and reflects associativity is \cite[N.8]{Dudzik_2016}. Thus any hypermagma morphism $f : (X, \star_X) \to (Y, \star_Y)$ for which $\star_X,\star_Y$ happen to be associative gives rise via Lemma \ref{prop:hmag-mag-morphisms} to a magma morphism ${F : (\pow(X),\ostar_X) \to (\pow(Y), \ostar_Y)}$ where $\ostar_X,\ostar_Y$ are associative, and vice-versa.
\end{proof}

Denoting by $\cat{uMag}(\pow(X),\pow(Y))_A$ the set of colax atom-preserving morphisms between unital magmas in $\cat{CABA}$, we also have the following:

\begin{proposition} \label{prop:uhmag-mag-morphisms}
    The assignment $\Phi$ of Proposition \ref{prop:hmag-mag-morphisms} restricts to a bijection $$ \cat{uHMag}(X,Y) \cong \cat{uMag}(\pow(X),\pow(Y))_A $$ which preserves composites and identities.
\end{proposition}

\begin{proof}
    With Proposition \ref{prop:hmag-mag-morphisms} in place, it suffices to show that unitality of morphisms is preserved and reflected by the bijection in Proposition \ref{prop:hmag-mag-morphisms}. We first show that if $f : (X, \star_X, 1_X) \to (Y, \star_Y, 1_Y)$ is a unital hypermagma morphism, then the map $\hat{f}$ constructed in Lemma \ref{prop:hmag-mag-morphisms} is a unital magma morphism: Since $f$ is unital, the singleton set $\{f(1_X)\}$ is an identity for $\pow(Y)$ under the induced operation $\ostar_Y$. For example, we have $$ A \ostar_Y \{f(1_X)\} = \bigcup_{a \in A} a \star_Y f(1_X) = \bigcup_{a \in A} a = A.$$ (Left-unitality is proved in exactly the same way.) By definition $\hat{f}(\{1_X\}) = \bigcup_{x \in \{1_X\}}f(x) = \{f(1_X)\}$, so we see that $\hat{f}$ is unital. 
    
    Conversely, if ${F : (\pow(X),\ostar_X,e_X) \to (\pow(Y), \ostar_Y, e_Y)}$ is a morphism of unital magmas (recalling that we require the unit element to be a singleton), we have $$F(e_X) \ostar_Y A = A = A \ostar_Y F(e_X)$$ for any subset $A$ of $Y$, so in particular this holds for singletons $A = \{y\}$. Denote $e_X = \{1_X\}$, and $e_Y = \{1_Y\}$. Since $s(F)(1_X) = F(\{1_X\}) = F(e_X)$, we immediately have that $s(F)(1_X)$ is an identity for $(Y,\star_Y,1_Y)$, and must therefore have that $s(F)(1_X) = 1_Y$; i.e., $s(F)$ is a unital hypermagma morphism.
\end{proof}

Denoting by $\cat{Mon}(\pow(X),\pow(Y))_A$ the set of colax, atom-preserving magma morphisms between unital associative magmas in $\cat{CABA}$, and recalling that unital associative magmas in $\cat{CABA}$ are alternatively characterized as unital quantales whose underlying set is a power set, we immediately have:

\begin{corollary}
    \label{prop:hmon-mon-morphisms}
    The assignment $\Phi$ of Proposition \ref{prop:hmag-mag-morphisms} restricts to a bijection $$\cat{HMon}(X,Y) \cong \cat{Mon}(\pow(X),\pow(Y))_A$$ which preserves composites and identities.
\end{corollary}

\subsection{Reversible hyperstructures as residuated lattices}

We recall the definition of a nonassociative relation algebra. For notational convenience, we give the definition of a residuated binary operation in terms of the conjugates, rather than the residuals, of the given binary operation; however, we note for completeness that since the conjugates and residuals of a given binary operation may be obtained from one another (conversion between the two is described in \cite{Jonsson_Tsinakis_1993}), our definition is equivalent to the more common one given in terms of residuals.

\begin{definition} \cite{Jonsson_Tsinakis_1993}
    A binary operation $\ostar$ on a suplattice $(A, \land, \lor, \bot, \top)$ is called \textbf{residuated} if there exist two further binary operations $\triangleleft, \triangleright$ on $A$ such that for all $x,y,z \in A$, we have $$ (x \ostar y) \land z = \bot \iff (x \triangleright z) \land y = \bot \iff (z \triangleleft y) \land x = \bot.$$ We call $A = (A_0,\ostar, \triangleleft, \triangleright)$ a \textbf{residuated Boolean algebra} if $A_0 = (A, \land, \lor, \bot, \top, (-)^c)$ is a Boolean algebra and $\ostar$ is residuated in the sense above. We call $A$ \textbf{unital} if there is an atom $e$ of $A$ such that $e \ostar x = x = x \ostar e$. We call a Boolean algebra $A$ a \textbf{nonassociative relation algebra} if $A$ is unital and is equipped with a unary operation $(-)'$ such that $A$ is residuated under $x \triangleright y := x' \ostar y$ and $x \triangleleft y := x \ostar y'$. A nonassociative relation algebra $(A, \ostar, e, \triangleleft, \triangleright)$ is called a \textbf{relation algebra} if $\ostar$ is associative.
\end{definition}

We show that a mosaic $X$ canonically gives rise to a nonassociative relation algebra structure on $\pow(X)$, and vice-versa. By an \textbf{involutive suplattice morphism}, we mean a suplattice morphism which is also an involution. Since any mosaic is equipped with an involution, we will need the following lemma to prove our main result:

\begin{lemma} \label{lem:involution-to-involution} An involution $(-)'$ on a set $X$ extends to an involutive suplattice morphism on $\pow(X)$, and an involutive suplattice morphism $(-)'$ on $\pow(X)$ gives rise to an involution on $X$.
\end{lemma}

\begin{proof}
    We may extend an involution $(-)'$ on a set $X$ to a map $(-)' : \pow(X) \to \pow(X)$ by sending \begin{equation} \label{eq:involution} A \mapsto A' := \{x' \in X : x \in A\}.\end{equation} It is easily checked that $(-)'$ is a monotone involution on $\pow(X)$ which preserves arbitrary joins.

    Conversely, any involutive suplattice morphism $(-)': \pow(X) \to \pow(X)$ is a bijection since it is an involution of sets. Thus, by Proposition \ref{prop:isomorphism-atom-preserving}, $(-)'$ sends atoms of $\pow(X)$ to atoms of $\pow(X)$. Restricting $(-)'$ to singletons therefore yields an involution on $X$, as desired.
\end{proof}

Finally, we show that mosaics correspond to nonassociative relation algebras.

\begin{theorem} \label{thm:is-relation-algebra}
    The hyperstructure $(X,\star,1_X,(-)')$ is a mosaic if and only if $(\pow(X),\ostar,e)$ is a nonassociative relation algebra.
\end{theorem}

\begin{proof} First, we suppose $(X,\star,1_X,(-)')$ is a mosaic, and show that $(\pow(X),\ostar,e,\triangleleft, \triangleright)$ with $$ A \triangleleft B := A \ostar B', \; A \triangleright B := A' \ostar B $$ is unital and residuated, where $(-)' : \pow(X) \to \pow(X)$ is the involutive suplattice morphism constructed in \ref{eq:involution}. 

It is easy to check that $e = \{1_X\}$ is a unit for $\ostar$ as a consequence of $1_X$ being an identity element for $X$, whence $(\pow(X),\ostar)$ is unital. To show that the three equalities \begin{align} \label{eq:residuated1} (A \ostar B) \land C &= \varnothing \\ \label{eq:residuated2}
        (A \triangleright C) \land B & = \varnothing \\ \label{eq:residuated3}
        (C \triangleleft B) \land A &= \varnothing \end{align} are equivalent for all $A,B,C$, first observe that $$(A \ostar B) \land C = \bigcup_{a \in A, b \in B} (a \star b) \cap C $$ is inhabited if and only if we have $c \in a \star b$ for some triple $(a,b,c) \in A \times B \times C$. Since $X$ is a mosaic, the latter statement is true if and only if both $b \in a' \star c$ and $a \in c \star b'$; i.e., if and only if both $$(A' \ostar C) \cap B = (A \triangleright C) \land B \text{ and } (C \ostar B') \cap A = (C \triangleleft B) \land A$$ are inhabited. Identical arguments prove that \ref{eq:residuated2} is equivalent to \ref{eq:residuated1} and \ref{eq:residuated3}, as well as that \ref{eq:residuated3} is equivalent to \ref{eq:residuated1} and \ref{eq:residuated2}, whence $(\pow(X), \ostar)$ is residuated.

        Conversely, suppose $(\pow(X),\ostar,e,\triangleleft,\triangleright, (-)')$ is a nonassociative relation algebra. It is easy to check, as in the proof of \cite[N.11]{Dudzik_2016}, that $(\pow(X), \ostar, e)$ is unital if and only if $X$ is unital (taking the identity element $1_X$ of $X$ to be the lone element of $e$). Moreover, Lemma \ref{lem:involution-to-involution} shows that the involution $(-)'$ preserves atoms, and so gives rise to an involution on $X$, which we also denote by $(-)'$. To show that $(X,\star,1_X,(-)')$ is a mosaic, it therefore suffices to check that it is reversible. But note that if the statements (\ref{eq:residuated1})-(\ref{eq:residuated3}) are equivalent for all $A,B,C \subset X$, then they are in particular equivalent for any three singleton subsets of $X$, which is exactly the statement of Remark \ref{rem:charfns}. We see that $X$ is reversible.
\end{proof}

Recalling Proposition \ref{prop:hmon-mon-morphisms}, along with \cite[N.16]{Dudzik_2016}, which shows that the assignment of \ref{eq:hyperop-to-binop} preserves and reflects commutativity, we immediately have the following.

\begin{corollary} \label{cor:commutative-and-hypergroup} There are bijections between
\begin{itemize}
    \item[i.] Commutative mosaics $(X,\star,1_X,(-)')$ and nonassociative relation algebras $(\pow(X),\ostar,e)$ where $\ostar$ is commutative.

    \item[ii.] Hypergroups $(X,\star,1_X,(-)')$ and relation algebras $(\pow(X),\ostar,e)$.

    \item[iii.] Canonical hypergroups $(X,\star,1_X,(-)')$ and commutative relation algebras $(\pow(X),\ostar,e)$.
\end{itemize} 
\end{corollary}

Let $\mathcal{H}$ denote any of the categories of hyperstructures listed in Definition \ref{def:gloss-hyper}, and for a given $\mathcal{H}$, let $\mathcal{O}_\mathcal{H}(\cat{CABA})$ denote the corresponding category of objects in CABA. We use $\mathcal{O}_\mathcal{H}(\cat{CABA})_A$ to denote the subcategory of $\mathcal{O}_\mathcal{H}(\cat{CABA})$ whose morphisms are the atom-preserving colax magma morphisms in $\cat{SupLat}$. With this notation, the results above may be summarized as follows.

\begin{theorem}
    For any class $\mathcal{H}$ of hyperstructures listed in Definition \ref{def:gloss-hyper} and its corresponding class $\mathcal{O}_\mathcal{H}(\cat{CABA})$ of CABAs with extra structure, there is an equivalence of categories $$\mathcal{H} \cong \mathcal{O}_\mathcal{H}(\cat{CABA})_A$$ and a pseudomonic embedding $$ \mathcal{O}_\mathcal{H}(\cat{CABA})_A \hookrightarrow \mathcal{O}_\mathcal{H}(\cat{CABA}).$$
\end{theorem}

\begin{proof} For a given $\mathcal{H}$, the equivalence $$ \mathcal{H} \cong \mathcal{O}_\mathcal{H}(\cat{CABA})_A$$ is instantiated by the functor $$\pow(-) := \begin{cases}
        (X, \star) \mapsto (\pow(X), \ostar) \\
        (f: X \to Y) \mapsto (\hat{f} : \pow(X) \to \pow(Y))
    \end{cases}$$ constructed in \ref{prop:hmag-mag-morphisms}.
    
    Well-definedness and essential surjectivity on objects in the case where $\mathcal{H}$ is $\cat{HMag}$ are both immediate consequences of Theorem \ref{thm:lattice-ops}. Well-definedness and essential surjectivity on objects in the cases where $\mathcal{H}$ is $\cat{HSemi}$ or $\cat{HMon}$ are \cite[N.8, N.11]{Dudzik_2016}. Noting that a mosaic is associative exactly when it is a hypergroup (this is \cite[Lemma 2.6]{mosaics}), the cases $\cat{Msc}$, $\cat{cMsc}$, $\cat{HGrp}$, and $\cat{Can}$ follow from Theorem \ref{thm:is-relation-algebra} and Corollary \ref{cor:commutative-and-hypergroup}. Since the morphisms of $\mathcal{O}_\mathcal{H}(\cat{CABA})_A$ for any $\mathcal{H}$ are atom-preserving colax magma morphisms, and additionally are unital in case the objects of $\mathcal{H}$ are unital, well-definedness on morphisms, together with the fact that $\pow(-)$ is fully faithful as a functor $\mathcal{H} \to \mathcal{O}_\mathcal{H}(\cat{CABA})_A$, follows from Propositions \ref{prop:hmag-mag-morphisms} and \ref{prop:uhmag-mag-morphisms}.
    
    Finally, to see that there is a pseudomonic embedding for any $\mathcal{H}$, note that the canonical inclusion $$\iota: \mathcal{O}_\mathcal{H}(\cat{CABA})_A \to \mathcal{O}_\mathcal{H}(\cat{CABA})$$ is obviously faithful. To see that it is full on isomorphisms, recall that any isomorphism of suplattices $f : \pow(X) \to \pow(Y)$ preserves atoms by \ref{prop:isomorphism-atom-preserving}. Thus, such an $f$ is its own image under the inclusion.
\end{proof}

\bibliographystyle{amsplain}
\bibliography{refs}

\end{spacing}

\end{document}